\documentclass[12pt,a4paper,reqno]{amsart}
\usepackage{geometry}
\usepackage{mathtools}
\usepackage{placeins}
\usepackage[numbers]{natbib}
\usepackage{amsmath}

\newtheorem{theorem}{Theorem}
\newtheorem{lemma}{Lemma}
\newtheorem{corollary}{Corollary}

\theoremstyle{remark}
\newtheorem*{remark}{Remark}

\numberwithin{equation}{section}

\begin{document}
\title[A summation of the form $\sum \zeta^{(\lowercase{n})} (\rho) X^{\rho}$]{A further generalisation of sums of higher derivatives of the Riemann Zeta Function}

\author{Andrew Pearce-Crump}
\address{Department of Mathematics, University of York, York, YO10 5DD, United Kingdom}
\email{aepc501@york.ac.uk}

\begin{abstract}
We prove an asymptotic for the sum of $\zeta^{(n)} (\rho)X^{\rho}$ where $\zeta^{(n)} (s)$ denotes the $n$th derivative of the Riemann zeta function, $X$ is a positive real and $\rho$ denotes a non-trivial zero of the Riemann zeta function. The sum is over the zeros with imaginary parts up to a height $T$, as $T \rightarrow \infty$.

We also specify what the asymptotic formula becomes when $X$ is a positive integer, highlighting the differences in the asymptotic expansions as $X$ changes its arithmetic nature.

%Let $\zeta^{(\lowercase{n})} (s)$ denote the $n$th derivative of the Riemann zeta function $\zeta(s)$. For a non-trivial zero $\rho$ of $\zeta(s)$ and a fixed positive number $X$ an asymptotic formula for $\sum_{0< \Im(\rho) \leq T} \zeta^{(\lowercase{n})} (\rho)X^{\rho}$ as $T \rightarrow \infty$ is given. We also specify what the asymptotic formula becomes when $X \geq 1, X \in \mathbb{Z}$, highlighting the differences in the asymptotic expansions as $X$ changes its arithmetic nature.
\end{abstract}

\maketitle

\section{Introduction} \label{Introduction}
Let $\zeta(s)$ denote the Riemann zeta function and let $\rho = \beta + i\gamma$ be a non-trivial zero of $\zeta (s)$. The starting point for the motivation of this paper is Landau's theorem \cite{Landau} which states that for any $X > 1$ we have
\begin{equation}
\sum_{0< \gamma \leq T} X^{\rho} = - \frac{T}{2 \pi} \Lambda (X) + O(\log T). \label{Landau}
\end{equation}
Here and throughout this article, $\Lambda (X)$ denotes the von Mangoldt function given by
\begin{equation}
\Lambda(X) =
\begin{cases}
\log p &\text{ if $X = p^k$ for some prime $p$ and some integer $k \geq 1$}\\
0 &\text{ otherwise}.\label{vonM}
\end{cases}
\end{equation}
Formula (\ref{Landau}) is proven by estimating the integral
\[
\frac{1}{2 \pi i} \int_{R} \frac{\zeta '}{\zeta}(s) X^s \ ds
\]
for a suitably chosen rectangle enclosing those zeros $\rho$ for which $0 < \gamma \leq T$.

Gonek \cite{GonekLandau} has generalised Landau's formula to be uniform in both variables. He showed that for $X,T>1$,
\begin{align}
\sum_{0< \gamma \leq T} X^{\rho} = & - \frac{T}{2 \pi} \Lambda (X) + O(X \log (2XT) \log\log (3X)) \nonumber \\
&+ O\left(\log X \min \left(T, \frac{X}{\langle X \rangle} \right) \right) + O\left(\log T \min \left(T, \frac{1}{\log X} \right) \right) \label{Gonek}
\end{align}
where $\langle X \rangle$ denotes the distance from $X$ to the nearest prime power other than $X$ itself. Note then that (\ref{Landau}) follows from (\ref{Gonek}) provided that $X$ is fixed and $T \rightarrow \infty$.

%A trivial estimate for this summation is $O(X T \log T)$ unconditionally and $O(\sqrt{X} T \log T)$ under the Riemann Hypothesis (referred to as RH from here in this article) which follows from the observation that the number of zeros with $0 < \gamma \leq T$ are $N(T) \sim \frac{T}{2 \pi} \log \left( \frac{T}{2 \pi} \right)$.
%Not necessary to spell out as the other terms are better (and it is an equality anyway)

The large number of error terms are explained in Gonek \cite{GonekLandau} which are all due to the different behaviour of the sum in different $X$ ranges. He then also notes that when $X=n \in \mathbb{N}$ and $T \gg n$, then the last two error terms of (\ref{Gonek}) are absorbed by the first error term and so (\ref{Gonek}) becomes
\begin{equation}
\sum_{0< \gamma \leq T} n^{\rho} = - \frac{T}{2 \pi} \Lambda (n) + O(n \log (2nT) \log\log (3n)). \label{KLP}
\end{equation}
The differences in formula (\ref{Gonek}) and formula (\ref{KLP}) highlights the differences in behaviour of sums of this form when $X$ is no longer an arbitrary positive real but is instead a positive integer, a theme that will reoccur several times in this paper.

Fujii has generalised Gonek's result (assuming RH) in \cite{FujiiLandau1} and \cite{FujiiLandau2} to find the sub-leading and sub-sub-leading behaviour. This is given by
\begin{equation}
\sum_{0< \gamma \leq T} X^{\frac{1}{2} + i\gamma} = - \frac{T}{2 \pi} \Lambda (X) + \frac{ X^{\frac{1}{2} + iT} \log \left( \frac{T}{2 \pi} \right)}{2 \pi i \log X} + X^{\frac{1}{2} + iT} \frac{1}{\pi} \arg \zeta \left(\frac{1}{2} + iT \right) + O \left( \frac{\log T}{(\log\log T)^2} \right). \label{Fujii1}
\end{equation}
%Further discussion of this is given in Fujii (CITE - see Fujii paper for references). - this line??

Fujii has also shown in \cite{FujiiShanks} that
\begin{equation}
\sum_{0< \gamma \leq T} \zeta ' (\rho) = \frac{T}{4 \pi} \log^2 \left( \frac{T}{2 \pi} \right) + (-1 + C_0) \frac{T}{2 \pi} \log \left( \frac{T}{2 \pi} \right) + (1 - C_0 - C_0^2 +3 C_1) \frac{T}{2 \pi} +E(T) \label{Fujii2}
\end{equation}
where $E(T)$ is given explicitly both unconditionally and assuming RH and where $C_0$ and $C_1$ are the constants in the expansion
\[
\zeta(s) = \frac{1}{s-1} + C_0 + C_1 (s-1) + ... .
\]

Fujii forms a hybrid of the sums (\ref{Landau}) and (\ref{Fujii2}) in his paper \cite{FujiiDist}. He gives the following explicit formula for a fixed positive real number $X$,
\begin{align}
\sum_{0< \gamma \leq T} \zeta ' (\rho) X^{\rho} &= - \Delta (X) \frac{T}{2 \pi} \Bigg(\log X \left(\frac{1}{2} \log \left( \frac{T}{2 \pi} \right) - \frac{1}{2} + \frac{\pi i}{4} \right) \nonumber -  \sum_{mr=X} \Lambda (r) \log m \Bigg) \nonumber \\
& + X \sum_{m \leq \frac{T}{2 \pi X}} e^{2 \pi i m X} \log^2 m  + \frac{1}{2} X \log X \sum_{m \leq \frac{T}{2 \pi X}} e^{2 \pi i m X} \log m \nonumber \\
&- \left( \frac{1}{2} X \log^2 X - \frac{\pi i}{4} X \log X \right) \sum_{m \leq \frac{T}{2 \pi X}} e^{2 \pi i m X} - X \sum_{mr \leq \frac{T}{2 \pi X}} e^{2 \pi i mr X} \Lambda (r) \log m \nonumber \\
&+ O(\sqrt{T} \log^3 T) \label{Fujii3}
\end{align}
where $\Delta (X)$ is defined by
\begin{equation}
\Delta(X) =
\begin{cases}
1 &\text{ if $n \in \mathbb{Z}$}\\
0 &\text{ otherwise}.\label{Delta} \\
      \end{cases}
\end{equation}
throughout this article.

Fujii then narrows down $X$ to be a positive integer and gives an asymptotic formula for this, again exhibiting the behaviour that Gonek noted in \cite{GonekLandau}, about how the arithmetic nature of $X$ can affect the asymptotic expansion. Specifically, he finds for $X \geq 1$, $X \in \mathbb{Z}$ that
\begin{align}
\sum_{0< \gamma \leq T} \zeta ' (\rho) X^{\rho} & = \frac{T}{4 \pi} \log^2 \left( \frac{T}{2 \pi} \right) + (-1 + C_0 - \log X) \frac{T}{2 \pi} \log \left( \frac{T}{2 \pi} \right) \nonumber \\
& + \frac{T}{2 \pi} \left( 1 - C_0 - C_0^2 + 3C_1 + \sum_{X=mn} \Lambda (n) \log m - \left( C_0 - 1 + \frac{1}{2} \log X \right) \log X  \right) \nonumber \\
& + O(T e^{-C \sqrt{\log T}}) \label{Fujii4}
\end{align}
where $C_0, C_1$ are given above and $C$ is a positive constant. Clearly setting $X=1$ then gives (\ref{Fujii2}).

A generalisation of Fujii's result has been given in Jakhlouti and Mazhouda \cite{JakMaz} to give an analogue of (\ref{Fujii3})
for Dirichlet $L$-functions. This extension is taken further to an asymptotic formula at the $L$-function's $a$-points, for any fixed complex number $a$. We state the analogue to (\ref{Fujii3}) which is their Lemma 3 with $a=0$,
\begin{multline*}
\sum_{0< \gamma_{\chi} \leq T} L' (\rho_{ \chi},\chi) X^{\rho_{\chi}}  = \\%-\frac{aT}{2 \pi} \log^2 \left( \frac{qT}{2 \pi} \right) + \frac{aT}{\pi} \log \left(\frac{qt}{2 \pi} \right) - \frac{aT}{\pi} \\
- \Delta (X) \chi [\Delta(X) X] \frac{T}{2 \pi} \Bigg( \log X \left(\frac{1}{2} \log \left( \frac{qT}{2 \pi} \right) - \frac{1}{2} + \frac{\pi i}{4}\right) - \sum_{mr=X} \Lambda (r) \log m \Bigg) \\
 + \frac{X}{\sqrt{q}} \sum_{m \leq \frac{qt}{2 \pi X}}  e^{2 \pi i m \frac{X}{q}} \overline{\chi} (m) \log^2 m + \frac{1}{2 \sqrt{q}} X \log X \sum_{m \leq \frac{qt}{2 \pi X}} e^{2 \pi i m \frac{X}{q}} \overline{\chi} (m) \log m \\
 - \left( \frac{1}{2 \sqrt{q}} X \log^2 X - \frac{\pi i}{4 \sqrt{q}} X \log X \right) \sum_{m \leq \frac{qt}{2 \pi X}} e^{2 \pi i m \frac{X}{q}} \overline{\chi} (m)  \\
 - \frac{X}{\sqrt{q}} \sum_{mr \leq \frac{qt}{2 \pi X}} e^{2 \pi i mr \frac{X}{q}} \Lambda (r) \overline{\chi} (r) \overline{\chi} (m) \log m  + O(\sqrt{T} \log^3 T)
\end{multline*}
as $T \rightarrow \infty$ for a fixed positive number $X$ and where $\chi$ is a primitive character mod $q$. Fixing $q=1$ gives (\ref{Fujii3}) and setting $q=1$ and $X=1$ gives (\ref{Fujii2}). %They do not appear to have restricted $X$ to be a positive integer which we will do later in this paper to further highlight how $X$ being an integer changes the form of these sorts of asymptotics??? or not?.

Other ideas have grown from (\ref{Landau}). For example, Ford and Zaharescu \cite{FordZah} start with Landau's theorem (\ref{Landau}) and investigate the distribution of the fractional parts of $\alpha \gamma$, where $\alpha$ is a fixed non-zero real number. This idea is then expanded upon by Ford, Soundararajan and Zaharescu in \cite{FordSoundZah}. Further examples of results that start with (\ref{Landau}) are found in papers \cite{GonekEmil} and \cite{KLP}.

\section{Statement of the Results} \label{Statement}
%Explicit Formula for $\sum \zeta^{(\lowercase{n})} (\rho) X^{\rho}$
Let $X$ be a fixed positive number. Let $\zeta^{({n})} (s)$ denote the $n$th derivative of the Riemann zeta function $\zeta (s)$ and let $\rho = \beta+i \gamma$ be a non-trivial zero of $\zeta(s)$. Write $s = \sigma + it$ with $\sigma, t \in \mathbb{R}$. We suppose that $T>T_0$ and that $T$ is not the imaginary part of the zeros of $\zeta (s)$. We further assume that $|T- \gamma| \gg \frac{1}{\log T}$. This restriction has no effect on the final result. Then we have the following result which is an analogue of Fujii's (\ref{Fujii3}).

\begin{theorem} \label{Explicit1}
For $X$ a fixed positive real number, we have
\begin{multline*}
\sum_{0< \gamma \leq T} \zeta^{(\lowercase{n})} (\rho)X^{\rho}  = \\ (-1)^n \left\{  \Delta(X) \frac{T}{2 \pi} \left( \log ^n X \left( \frac{1}{2} \log \left( \frac{T}{2 \pi} \right) - \frac{1}{2} + \frac{\pi i}{4} \right) - \sum_{mr=X} \Lambda (r) \log^n m  \right) \right.\\
+ X \log^n X \left(\frac{1}{2} \log X - \frac{\pi i}{4} \right) \sum_{m \leq \frac{T}{2 \pi X}} e^{2 \pi i m X} + \frac{X}{2} \log^n X \sum_{m \leq \frac{T}{2 \pi X}} e^{2 \pi i m X} \log m  \\
- X \sum_{mr \leq \frac{T}{2 \pi X}}  e^{2 \pi i mr X} \Lambda (r) \log^n (rX) \Bigg\} + O(\sqrt{T} \log^{n+2} T),
\end{multline*}
where
\[
\Delta(X) =
\begin{cases}
1 &\text{ if $n \in \mathbb{Z}$}\\
0 &\text{ otherwise}.\\
      \end{cases}
\]
\end{theorem}

\begin{remark}
Setting $n=1$ in the theorem recovers Fujii's result in (\ref{Fujii3}).
\end{remark}

When we restrict $X$ to being a positive integer we obtain a special case of the above results. It is evident from the statement of the following corollary that we see the changing behaviour of our asymptotic expansions depending on whether $X>0$ is any real number and when $X \geq 1$ is an integer.
% as follows. %Insert comment here about how changing the nature of X changes how the result looks - see Fujii for ideas here.

\begin{corollary} \label{IntegerCase} %Need to state clearly what C_j and A_j are here as we reference it later.
If $X \geq 1$, $X \in \mathbb{Z}$, then 
\begin{multline*}
\sum_{0< \gamma \leq T} \zeta^{(\lowercase{n})} (\rho) X^{\rho}  = \\
 (-1)^{n+1} \frac{T}{2 \pi} \log ^n X \Bigg\{  \sum_{k=0}^{n} \sum_{u=0}^{k+1} {n \choose k} {{k+1} \choose u} (-1)^u \frac{1}{k+1}  \log^{k+1-u} \left( \frac{T}{2 \pi} \right)\log^{u-k} X \\
+ \sum_{k=0}^{n}  \sum_{l=0}^{k} \sum_{u=0}^{k-l}  {n \choose k} {k \choose l} {{k-l} \choose u} (-1)^l l! \left( -1 + \sum_{j=0}^{l} (-1)^j C_j \right) \log^{k-l-u} \left( \frac{T}{2 \pi} \right) \log^{u-k} X \\
+ \sum_{k=0}^{n} {n \choose k} (-1)^{k+1} k! A_k \log ^{-k} X - \left( \log \left( \frac{T}{2 \pi} \right) -1 - \sum_{mr=X} \Lambda (r) \log^n m \right) \Bigg\} + E_n(T).
\end{multline*}
where
\[
E_n(T) = O \left(T  \mathrm{e}^{-C \sqrt{\log T}} \right)
\]
with $C$ is a positive constant. If we assume the Riemann Hypothesis, then
\[
E_n(T) = O \left(T^{\frac{1}{2}} \log^{n+2} T \right).
\]
Further, the $C_j$ are the coefficients in the Laurent expansion for $\zeta (s)$ about $s=1$, given by
\[
\zeta(s) = \frac{1}{s-1} + \sum_{j=0}^\infty C_j (s-1)^j
\]
and the $A_j$ are the coefficients in the Laurent expansion for $\frac{\zeta ' (s)}{\zeta (s)}$ about $s=1$, given by
\[
\frac{\zeta ' (s)}{\zeta (s)} = - \frac{1}{s-1} + \sum_{j=0}^\infty A_j (s-1)^j.
\]
\end{corollary}

\begin{remark}
Note that the $A_j$ are related to the $C_j$ by the following recursive formula, as shown in Israilov \cite{Israilov}, given by
\[
A_j =
\begin{cases}
C_0 &\text{ if $j=0$}\\
(j+1)C_j - \sum_{k=0}^{j-1} A_k C_{j-1-k} &\text{ if $j \geq 1$}.\\
\end{cases}
\]
\end{remark}

\begin{remark}
Setting $n=1$ in the corollary recovers Fujii's result in (\ref{Fujii4}). Setting $n=1$ and $X=1$ recovers Fujii's result in (\ref{Fujii2}). Setting $X=1$ recovers our result from \cite{HPC} for general $n$, given later on in this paper as Theorem \ref{General SC}.
\end{remark}

\section{Outline of the Paper}
So far we have described the motivation for studying asymptotic expansions of the sums given in Section \ref{Introduction}. We have stated the main results in Section \ref{Statement} that we will prove in the following sections.

In Section \ref{Prelims} we will recall some basic facts about the Riemann zeta function $\zeta (s)$ and the Riemann xi function $\xi (s)$. One of the main results that we will need to state is the Theorem from Hughes and Pearce-Crump \cite{HPC} that will be essential in proving Corollary~\ref{IntegerCase}.

In Section \ref{ProofThm} we will use the tools from Section \ref{Prelims} to prove Theorem \ref{Explicit1}. The integral we use to prove this theorem is given by

\begin{equation}
I = \frac{1}{2 \pi i} \int_R \frac{\xi '}{\xi} (s) \zeta^{(n)} (s) X^s \ ds
\end{equation}
where $\xi (s)$ is the Riemann xi function, $\zeta^{(n)} (s)$ denotes the $n$th derivative of $\zeta (s)$ and R denotes the rectangular positively oriented contour with the vertices are $c+i, c+iT, 1-c+it, 1-c+i$ connected in this order and $c = 1 + \frac{1}{\log T}$. The non-trivial zeros of $\zeta (s)$ up to a height $T$ are contained within $R$ and so by Cauchy's Theorem the integral represents the summation

\begin{equation}
I = \sum_{0 < \gamma \leq T}  \zeta^{(n)} (\rho) X^{\rho}.
\end{equation}

We split this section up into several subsections, each corresponding to different parts of the contour that we will be integrating over, to show that most of the contribution comes from the left-hand side of the contour, while the other sides mostly only contribute to the error term.

Finally in Section \ref{ProofCor} we will prove Corollary \ref{IntegerCase} which will highlight the differences between the general case proved in Section \ref{ProofThm} for any positive number $X$ and the case when $X$ is a positive integer. As described above we will use a result from \cite{HPC} to do most of the work here. This will again highlight the observation of Gonek's in \cite{GonekLandau} that the asymptotic formulae tend to change quite dramatically depending on the arithmetic nature of $X$.

\section{Preliminary Lemmas} \label{Prelims}
In this section we recall some basic information about $\zeta (s)$ and $\xi (s)$, as well as recalling some results from other papers that will be useful in our proof. Any facts that are not explicitly referenced in this section can be found in any good text about the Riemann zeta function, for example they can be found in Titchmarsh \cite{Titch}.

Firstly recall that the functional equation for $\zeta (s)$ is given by
\begin{equation*}
\zeta (s) = \chi (s) \zeta (1-s)
\end{equation*}
where
\begin{equation*}
\chi (s) = 2^s \pi^{s-1} \sin \left( \frac{\pi s}{2} \right) \Gamma (1-s)
\end{equation*}
and $\Gamma (s)$ is the Gamma function throughout this paper. We state a more general functional equation for $\zeta ^{(n)} (s)$ that is proved using the functional equation for $\zeta (s)$ and the Leibniz product rule.

\begin{lemma}\label{functional equation zeta (s)}
The general functional equation for $\zeta ^{(n)} (s)$ is given by the following formula
\begin{equation}
\zeta ^{(n)} (s) = \frac{1}{\chi (1-s)} \sum_{k=0}^n \binom{n}{k} (-1)^k \frac{\chi^{(n-k)} (s)}{\chi(s)} \zeta^{(k)} (1-s). \label{eq:GeneralFuncEqu1}
\end{equation}
\end{lemma}

We now recall that for $\sigma >1$ we may write both $\zeta^{(n)} (s)$ and $\frac{\zeta '}{\zeta} (s)$ in terms of their Dirichlet series which are given by
\begin{equation}
\zeta^{(n)} (s) = (-1)^n \sum_{m=1}^{\infty} \frac{\log^n m}{m^s} \label{ZetaDeriv}
\end{equation}
and
\begin{equation}
\frac{\zeta'}{\zeta} (s) = - \sum_{r=1}^{\infty} \frac{\Lambda (r)}{r^s} \label{ZetaLogDeriv},
\end{equation}
where $\Lambda (r)$ is the von Mangoldt function defined in (\ref{vonM}).

We will also need the following result that we proved in \cite[Sect.5]{HPC} for the integral along the top and the bottom of our contour.
\begin{lemma} \label{main S^T}
For $c = 1 + \frac{1}{\log T}$, we have
\begin{equation*}
\int_{1-c}^{c} |\zeta ^{(n)} (\sigma +iT) | \ d\sigma \ll T^{\frac{1}{2}} \log^n T. \label{maintop}
\end{equation*}
\end{lemma}

For the Riemann xi function we write
\begin{equation}
\xi (s) = \frac{1}{2} s (s-1) \pi^{-\frac{s}{2}} \Gamma \left( \frac{s}{2} \right) \zeta(s)   \label{xi}
\end{equation}
so the functional equation for $\xi (s)$ is given by
\begin{equation*}
\xi (s) = \xi (1-s).
\end{equation*}

We now observe that
\begin{equation}
\frac{\xi '}{\xi} (s) = \frac{2s -1}{s(s-1)} - \frac{1}{2} \log \pi + \frac{1}{2} \psi \left( \frac{s}{2} \right) + \frac{\zeta '}{\zeta} (s) \label{LogDerivXi1}
\end{equation}
where we have written
\begin{equation*}
\psi (s) = \frac{\Gamma '}{\Gamma} (s)
\end{equation*}
and for $|\arg s| < \pi - \delta$ with arbitrarily fixed positive $\delta$ and for $|s| \geq \frac{1}{2}$ we have
\begin{equation*}
\psi (s) = \log s + O \left( \frac{1}{|s|} \right) = \log t + \frac{\pi i}{2} + O \left( \frac{1}{t} \right)
\end{equation*}
as $t \rightarrow \infty$.

Combining these two observations, we have the following Lemma.

\begin{lemma} \label{LemmaLogDerivXi}
With the conditions written above, we have
\begin{equation}
\frac{\xi '}{\xi} (s) = \frac{1}{2} \log \left( \frac{t}{2 \pi} \right) + \frac{\pi i}{4} + \frac{\zeta '}{\zeta} (s) + O \left( \frac{1}{t} \right). \label{LogDerivXi2}
\end{equation}
\end{lemma}

Finally to prove Corollary \ref{IntegerCase} we will need the main result from \cite{HPC} which we state in fullness here for ease of reference.
\begin{theorem}\label{General SC}
With the setup as stated in our main results Section \ref{Statement} and with the $C_j$ and $A_j$ coefficients defined in Corollary \ref{IntegerCase}, we have
\begin{align*}
\sum_{0 < \gamma \leq T} \zeta^{(n)} (\rho) &=
 (-1)^{n+1} \frac{1}{n+1} \frac{T}{2 \pi} \log^{n+1} \left( \frac{T}{2 \pi} \right)   \\
&+ (-1)^{n+1} \sum_{k=0}^n \binom{n}{k}  (-1)^ {k}  k!  \left(-1 + \sum_{j=0}^k (-1)^j C_j \right) \frac{T}{2 \pi} \log^{n-k} \left( \frac{T}{2 \pi} \right) \\
&+ n! A_n  \frac{T}{2 \pi}  + E_n(T)
\end{align*}
where unconditionally,
\[
E_n(T) = O \left(T  \mathrm{e}^{-C \sqrt{\log T}} \right)
\]
with $C$ is a positive constant. If we assume the Riemann Hypothesis, then
\[
E_n(T) = O \left(T^{\frac{1}{2}} \log^{n+2} T \right).
\]
\end{theorem}

\section{Proof of Theorem \ref{Explicit1}} \label{ProofThm}
Let $X$ be a fixed positive real. We write $s=\sigma + it$ with $\sigma, t \in \mathbb{R}$ and $\rho = \beta + i \gamma$ for a non-trivial zero of the Riemann zeta function $\zeta (s)$. Suppose $T>T_0$ and $T$ is not the imaginary part of the zeros of $\zeta (s)$ and further that $|T - \gamma| \gg \frac{1}{\log T}$ where $\gamma$ is the imaginary part of any zero $\rho$. This restriction on $T$ is harmless within our remainder terms.
%changed 'let $X$ be a fixed positive number to ... be a fixed positive real here

Set $c = 1 + \frac{1}{\log T}$ and consider the integral
\begin{equation}
I = \frac{1}{2 \pi i} \int_R \frac{\xi '}{\xi} (s) \zeta^{(n)} (s) X^s \ ds \label{MainInt}
\end{equation}
where $\xi (s)$ is the Riemann xi function, $\zeta^{(n)} (s)$ denotes the $n$th derivative of $\zeta (s)$ and R denotes the rectangular positively oriented contour with vertices given by $c+i, c+iT, 1-c+it, 1-c+i$, connected in this order.

By Cauchy's Theorem,
\begin{equation}
I = \sum_{0 < \gamma \leq T}  \zeta^{(n)} (\rho) X^{\rho}. \label{Cauchy}
\end{equation}

We now need to evaluate $I$ in another way to obtain our asymptotic expansion. Decomposing the integral (\ref{MainInt}) along the sides of the contour as
\begin{align*}
I &= \frac{1}{2 \pi i} \left(\int_{1-c+i}^{c+i} + \int_{c+i}^{c+iT} + \int_{c+iT}^{1-c+iT} + \int_{1-c+iT}^{1-c+i}  \right) \frac{\xi '}{\xi} (s) \zeta^{(n)} (s) X^s \ ds \\
& = I_B + I_R + I_T + I_L.
\end{align*}

\subsection{Bounding $I_B$ and $I_T$} \hfill\\
Notice that by our choice of $T$ we may bound $I_B$ and $I_T$ trivially within our error. To do this, recall that for $-1 \leq \sigma \leq 2$ and with our general assumptions we have from Gonek \cite[Sect.2, p.127]{Gonek} that
%CHECK THIS REFERENCE
\[
\frac{\xi '}{\xi} (s) \ll \log^2 T.
\]
By Lemma \ref{main S^T}, we have
\begin{equation}
I_B + I_T \ll \log^2 T \int_{1-c}^{c} | \zeta^{(n)} (\sigma + it) X^{\sigma + it} | \ d\sigma \ll \sqrt{T} \log^ {n+2} T \label{TopBot}.
\end{equation}

\subsection{Evaluating $I_R$} \hfill\\
Writing $s = c+ it$ and using (\ref{LogDerivXi1}) we may write
\begin{align*}
I_R & = \frac{1}{2 \pi} \int_{1}^{T} \frac{\xi '}{\xi} (c+it) \zeta^{(n)} (c+it) X^{c+it} \ dt \\
& = \frac{1}{2 \pi} \int_{1}^{T} \left( \frac{2(c+it) -1}{(c+it)(c+it-1)} - \frac{1}{2} \log \pi + \psi \left(\frac{c+it}{2} \right) + \frac{\zeta '}{\zeta} (c+it) \right) \zeta^{(n)} (c+it) X^{c+it} \ dt \\
\end{align*}
Using the Dirichlet series (\ref{ZetaDeriv}),(\ref{ZetaLogDeriv}) and Lemma \ref{LemmaLogDerivXi} we may rewrite this as
\begin{align*}
I_R & = \frac{(-1)^n}{2 \pi} \int_{1}^{T} \left( \frac{1}{2} \log \left(\frac{t}{2 \pi} \right) + \frac{\pi i}{4} - \sum_{r=1}^{\infty} \frac{\Lambda (r)}{r^{c+it}} + O \left(\frac{1}{t}\right) \right) \sum_{m=1}^{\infty} \frac{\log^n m}{m^{c+it}} X^{c+it} \ dt \\
& = \frac{(-1)^n}{2 \pi} X^c \sum_{m=1}^{\infty} \frac{\log^n m}{m^{c}}  \int_{1}^{T} \left( \frac{1}{2} \log \left(\frac{t}{2 \pi} \right) + \frac{\pi i}{4} + O \left(\frac{1}{t}\right) \right) \left( \frac{X}{m} \right)^{it} \ dt \\
&\qquad+ \frac{(-1)^{n+1}}{2 \pi} X^c \sum_{m=1}^{\infty} \frac{\log^n m}{m^{c}} \sum_{r=1}^{\infty} \frac{\Lambda (r)}{r^{c}} \int_{1}^{T} \left( \frac{X}{mr} \right)^{it} \ dt \\
& = I_{R,1} + I_{R,2}.
\end{align*}

Consider $I_{R,1}$ first. We have (with $\Delta (X)$ defined as in (\ref{Delta})),
\begin{align*}
     I_{R,1} &= \frac{(-1)^n}{2\pi} X^c \sum_{m=1}^{\infty} \frac{\log ^n m}{m^{c}} \int_{1}^{T} \left( \frac{1}{2} \log \left( \frac{t}{2 \pi} \right) \ + \frac{ \pi i}{4} + O \left( \frac{1}{t} \right) \right) \left( \frac{X}{m} \right)^{it} \ dt \\
     & = \frac{(-1)^n}{2\pi} \Delta (X) \log ^n X \int_{1}^{T} \left( \frac{1}{2} \log \left( \frac{t}{2 \pi} \right) \ + \frac{ \pi i}{4} \right) \ dt \\
     & \quad + \frac{(-1)^n}{2\pi} X^c \sum_{\substack{m=1 \\ m \neq X}}^{\infty} \frac{\log ^n m}{m^{c}} \int_{1}^{T} \left( \frac{1}{2} \log \left( \frac{t}{2 \pi} \right) \ + \frac{ \pi i}{4} \right) \left( \frac{X}{m} \right)^{it} \ dt \\
     & \qquad + O(\log ^{n+2} T) \\
     & = I_{R,1,1} + I_{R,1,2} + O(\log ^{n+2} T).
\end{align*}

Firstly,
\begin{equation*}
    I_{R,1,1} = (-1)^n \Delta (X) \log ^n X \frac{T}{2 \pi} \left( \frac{1}{2} \log \left( \frac{T}{2 \pi} \right) - \frac{1}{2} + \frac{ \pi i}{4} \right) + O(1).
\end{equation*}

Next, integrating by parts and summing we obtain
\begin{align*}
    I_{R,1,2} &\ll
    X^c \sum_{\substack{m=1 \\ m \neq X}}^{\infty} \frac{\log ^n m}{m^{c}} \frac{\log T}{\log \left| \frac{X}{m} \right|}  \ll \log ^{n+2} T
\end{align*}

If $0<X<1$ we can do slightly better than this error term. Recombining, we have
\begin{equation}
    I_{R,1} = (-1)^n \Delta (X) \log ^n X \frac{T}{2 \pi} \left( \frac{1}{2} \log \left( \frac{T}{2 \pi} \right) - \frac{1}{2} + \frac{ \pi i}{4} \right) + O(\log ^{n+2} T). \label{I{R,1}}
\end{equation}

Now consider $I_{R,2}$. We have
\begin{align*}
   I_{R,2} & = \frac{(-1)^{n+1}}{2\pi} X^c \sum_{m=1}^{\infty} \frac{\log ^n m}{m^{c}} \sum_{r=1}^{\infty} \frac{\Lambda (r)}{r^{c}} \int_{1}^{T} \left( \frac{X}{mr} \right)^{it} \ dt \\
   & =  \frac{(-1)^{n+1}}{2\pi} X^c \sum_{k=1}^{\infty} \frac{1}{k^{c}} \sum_{k=mr} \log ^n m \ \Lambda (r) \int_{1}^{T} \left( \frac{X}{k} \right) ^{it} \ dt \\
   & = \frac{(-1)^{n+1}}{2\pi} \Delta (X) \sum_{X=mr} \log ^n m \ \Lambda (r) \int_{1}^{T} 1 \ dt \\
   & \quad + \frac{(-1)^{n+1}}{2\pi} X^c \sum_{\substack{k=1 \\ k \neq X}}^{\infty} \frac{1}{k^c} \sum_{k=mr} \log ^n m \ \Lambda (r) \int_{1}^{T} \left( \frac{X}{k} \right)^{it} \ dt \\
   & = I_{R,2,1} + I_{R,2,2}.
\end{align*}
Then
\begin{equation*}
    I_{R,2,1} = (-1)^{n+1} \Delta (X) \frac{T}{2 \pi} \sum_{X=mr} \log ^n m \ \Lambda (r) + O(1)
\end{equation*}
and as with the case for $I_{R,1,2}$ above, we have
\begin{align*}
    I_{R,2,2} &\ll
    X^c \sum_{\substack{k=1 \\ k \neq X}}^{\infty} \frac{1}{k^{c}} \sum_{k=mr} \log ^n m \ \Lambda (r) \frac{1}{\log |X/k|}  \ll \log ^{n+2} T
\end{align*}
where again the error can be improved slightly for $0<X<1$. Combining, we have
\begin{equation}
    I_{R,2} = (-1)^{n+1} \Delta (X) \frac{T}{2 \pi} \sum_{X=mr} \log ^n m \ \Lambda (r) + O( \log^ {n+2} T). \label{I{R,2}}
\end{equation}
Finally, we may combine (\ref{I{R,1}}) and (\ref{I{R,2}}) to obtain $I_R$, given by
\begin{equation}
    I_R = (-1)^n \Delta (X) \frac{T}{2 \pi} \left( \log^n X \left\{ \frac{1}{2} \log \left( \frac{T}{2 \pi} \right) - \frac{1}{2} + \frac{ \pi i}{4} \right\} - \sum_{X=mr} \log ^n m \ \Lambda (r) \right) + O (\log ^{n+2} T). \label{Right}
\end{equation}

\subsection{Evaluating $I_L$} \hfill\\
Finally we evaluate $I_L$, which is where most of the contribution to the asymptotic expansion comes from. Using Lemma \ref{functional equation zeta (s)} we have
\begin{align*}
I_L & = -\frac{1}{2 \pi i} \int_{1-c+i}^{1-c+iT}  \frac{\xi '}{\xi} (s) \zeta^{(n)} (s) X^s \ ds \\
& = -\frac{1}{2 \pi i} \int_{1-c+i}^{1-c+iT} \left( - \frac{\xi '}{\xi} (1-s) \right) \left( \frac{1}{\chi (1-s)} \sum_{k=0}^{n} {n \choose k} (-1)^k \frac{\chi^{(n-k)} (s)}{\chi(s)} \zeta^{(k)} (1-s) \right)  X^s \ ds \\
& = -\frac{1}{2 \pi} \int_{1}^{T} \left( - \frac{\xi '}{\xi} (c-it) \right)  \left( \frac{1}{\chi (c-it)} \sum_{k=0}^{n} {n \choose k} (-1)^k \frac{\chi^{(n-k)} (1-c-it)}{\chi(1-c-it)} \zeta^{(k)} (c-it) \right)  X^{1-c+it} \ dt.
\end{align*}
By complex conjugation we get
\begin{align*}
\overline{I_L} & = \frac{X^{1-c}}{2 \pi} \int_{1}^{T} \frac{\xi '}{\xi} (c+it) \left( \frac{1}{\chi (c+it)} \sum_{k=0}^{n} {n \choose k} (-1)^k \frac{\chi^{(n-k)} (1-c+it)}{\chi(1-c+it)} \zeta^{(k)} (c+it) \right)  X^{-it} \ dt \\
& = \frac{X^{1-c}}{2 \pi} \int_{1}^{T} \left\{ \frac{1}{2} \log \left( \frac{t}{2 \pi} \right) + \frac{\pi i}{4} - \sum_{r=1}^{\infty} \frac{\Lambda (r)}{r^{c+it}} + O\left( \frac{1}{t} \right) \right\} \\
&\qquad \left( \chi (1-c-it) \sum_{k=0}^{n} {n \choose k} (-1)^k \left( (-1)^{n-k} \log^{n-k} \left( \frac{t}{2 \pi} \right) + O \left( \frac{1}{t} \right) \right) (-1)^k \sum_{m=1}^{\infty} \frac{\log^k m}{m^{c+it}} \right)  X^{-it} \ dt \\ %Fix the formatting here
& = \frac{X^{1-c}}{2 \pi} \int_{1}^{T} \left\{ \frac{1}{2} \log \left( \frac{t}{2 \pi} \right) + \frac{\pi i}{4} - \sum_{r=1}^{\infty} \frac{\Lambda (r)}{r^{c+it}} \right\} \\
&\qquad \left( \chi (1-c-it) \sum_{k=0}^{n} {n \choose k} (-1)^{n+k} \log^{n-k} \left( \frac{t}{2 \pi} \right) \sum_{m=1}^{\infty} \frac{\log^k m}{m^{c+it}} \right)  e^{-it \log X} \ dt + O(\log^{n+2} T) \\ %Fix the formatting here
\end{align*}
where the second line follows from Lemma \ref{LemmaLogDerivXi}.

We now split this integral and evaluate each part separately.
\begin{align*}
\overline{I_L} & = \frac{X^{1-c}}{2 \pi} \frac{1}{2} \int_{1}^{T} \chi (1-c-it) \sum_{k=0}^{n} {n \choose k} (-1)^{n+k} \log^{n-k+1} \left( \frac{t}{2 \pi} \right) \sum_{m=1}^{\infty} \frac{\log^k m}{m^{c+it}} e^{-it \log X} \ dt \\
& + \frac{X^{1-c}}{2 \pi} \frac{\pi i}{4}  \int_{1}^{T} \chi (1-c-it) \sum_{k=0}^{n} {n \choose k} (-1)^{n+k} \log^{n-k} \left( \frac{t}{2 \pi} \right) \sum_{m=1}^{\infty} \frac{\log^k m}{m^{c+it}} e^{-it \log X} \ dt \\
& + \frac{X^{1-c}}{2 \pi}  \int_{1}^{T} \chi (1-c-it) \sum_{r=1}^{\infty} \frac{\Lambda (r)}{r^{c+it}}  \sum_{k=0}^{n} {n \choose k} (-1)^{n+k+1} \log^{n-k} \left( \frac{t}{2 \pi} \right) \sum_{m=1}^{\infty} \frac{\log^k m}{m^{c+it}} e^{-it \log X} \ dt \\
& + O(\log^{n+2} T) \\
& = J_1 + J_2 + J_3 + O(\log^{n+2} T).
\end{align*}

The key component to this section of the proof here is the method of stationary phase. Applications of this method to these types of problems can be in Gonek \cite[Sect.4, p.131]{Gonek}, in Levinson \cite{Levinson} and in Jakhlouti and Mazhouda \cite[Sect.2, p.13]{JakMaz}, amongst other places.

In an entirely analogous way to the proof that Gonek writes in \cite{Gonek}, we are able to prove the following result which we need to evaluate the integrals $J_1, J_2, J_3$.

\begin{lemma} \label{Gonek 5}
%check consistency of notation in this lemma with the rest of the paper.
%also check the order of what I have written in the integrals and the answers to make it consistent.
%do the sim thing to keep the alignment correct in the wording of the lemma
%may need to reword J1,J2,J3 slightly
Let $X$ be a fixed positive real. Let $\{b_m\}_{m=1}^{\infty}$ be a sequence of complex numbers such that for any $\epsilon >0, b_m \ll m^\epsilon$. Let $c>1$ and let $k \geq 0$ be an integer. Then for $T$ sufficiently large, we have
\begin{align*}
    \frac{1}{2 \pi} & \int_{1}^{T} \chi (1-c-it) \left( \sum_{m=1}^{\infty} b_m m^{-c-it} \right) \log ^k \left( \frac{t}{2 \pi} \right) e^{-it \log X} \ dt \\
    & = X^c \sum_{1 \leq m \leq \frac{T}{2 \pi X}} b_m \log ^k (mX) \ e^{-2 \pi i m X} + O \left( T^{c - \frac{1}{2}} \log ^k T \right).
\end{align*}
\end{lemma}

Applying Lemma \ref{Gonek 5} to each of the summations $J_k$, $k=1,2,3$, we have
\begin{align*}
J_1 & = \frac{X^{1-c}}{2 \pi} \frac{1}{2}  \sum_{k=0}^{n} {n \choose k} (-1)^{n+k} \sum_{m=1}^{\infty} \frac{\log^k m}{m^c}  \int_{1}^{T}  \chi (1-c-it)  \log^{n-k+1} \left( \frac{t}{2 \pi} \right) e^{-it \log (mX)} \ dt \\
& = \frac{X}{2}  \sum_{k=0}^{n} {n \choose k} (-1)^{n+k} \sum_{m \leq \frac{T}{2 \pi X}} \log^k (m) \log^{n-k+1} (mX) \ e^{-2 \pi i mX} + O(\sqrt{T} \log^{n+2} T) \\
& = \frac{X}{2}  \sum_{m \leq \frac{T}{2 \pi X}} e^{-2 \pi i mX} \left[ \sum_{k=0}^{n} {n \choose k} (-1)^{n-k} \log^k (m) \log^{n-k} (mX) \right] (\log (mX)) + O(\sqrt{T} \log^{n+2} T) \\
& =  \frac{X}{2}  \sum_{m \leq \frac{T}{2 \pi X}} e^{-2 \pi i mX} \log (mX) \left[ -\log (mX) + \log m \right]^n + O(\sqrt{T} \log^{n+2} T) \\
& = (-1)^n \frac{X}{2}  \sum_{m \leq \frac{T}{2 \pi X}} e^{-2 \pi i mX} \log (mX) \log^n X + O(\sqrt{T} \log^{n+2} T) \\
& = (-1)^n \frac{X}{2}  \log^{n+1} X \sum_{m \leq \frac{T}{2 \pi X}} e^{-2 \pi i mX} + (-1)^n \frac{X}{2}  \log^n X \sum_{m \leq \frac{T}{2 \pi X}} e^{-2 \pi i mX} \log m + O(\sqrt{T} \log^{n+2} T).
\end{align*}

Similarly,
\begin{align*}
J_2 & = \frac{X^{1-c}}{2 \pi} \frac{\pi i}{4}  \sum_{k=0}^{n} {n \choose k} (-1)^{n+k} \sum_{m=1}^{\infty} \frac{\log^k m}{m^c}  \int_{1}^{T}  \chi (1-c-it)  \log^{n-k} \left( \frac{t}{2 \pi} \right) e^{-it \log (mX)} \ dt \\
& = \frac{\pi i}{4} X \sum_{k=0}^{n} {n \choose k} (-1)^{n+k} \sum_{m \leq \frac{T}{2 \pi X}} \log^k (m) \log^{n-k} (mX) \ e^{-2 \pi i mX} + O(\sqrt{T} \log^{n+1} T) \\
& = \frac{\pi i}{4} X \sum_{m \leq \frac{T}{2 \pi X}} e^{-2 \pi i mX} \left[ \sum_{k=0}^{n} {n \choose k} (-1)^{n+k} \log^k (m) \log^{n-k} (mX) \right] + O(\sqrt{T} \log^{n+1} T) \\
& = \frac{\pi i}{4} X \sum_{m \leq \frac{T}{2 \pi X}} e^{-2 \pi i mX} \left[ \log (m) - \log (mX) \right]^n + O(\sqrt{T} \log^{n+1} T) \\
& = (-1)^n \frac{\pi i}{4} X \log^n X \sum_{m \leq \frac{T}{2 \pi X}} e^{-2 \pi i mX} + O(\sqrt{T} \log^{n+1} T).
\end{align*}

Finally,
\begin{align*}
J_3 & = \frac{X^{1-c}}{2 \pi} \sum_{k=0}^{n} {n \choose k} (-1)^{n+k+1} \sum_{r=1}^{\infty} \frac{\Lambda (r)}{r^c} \sum_{m=1}^{\infty} \frac{\log^k m}{m^c}  \int_{1}^{T}  \chi (1-c-it)  \log^{n-k} \left( \frac{t}{2 \pi} \right) e^{-it \log (mrX)} \ dt \\
& = X \sum_{k=0}^{n} {n \choose k} (-1)^{n+k+1} \sum_{mr \leq \frac{T}{2 \pi X}} \Lambda (r) \log^k (m) \log^{n-k} (mrX) \ e^{-2 \pi i mrX} + O(\sqrt{T} \log^{n+2} T) \\
& = - X \sum_{mr \leq \frac{T}{2 \pi X}}  e^{-2 \pi i mrX} \Lambda (r) \left[\sum_{k=0}^{n} {n \choose k} (-1)^{n+k}  \log^k (m) \log^{n-k} (mrX) \right] + O(\sqrt{T} \log^{n+2} T) \\
& = - X \sum_{mr \leq \frac{T}{2 \pi X}} e^{-2 \pi i mrX} \Lambda (r)  \left[ -\log(mrX) + \log m \right]^n + O(\sqrt{T} \log^{n+2} T) \\
& = (-1)^{n+1} X \sum_{mr \leq \frac{T}{2 \pi X}} e^{-2 \pi i mrX} \Lambda (r) \log^n (rX) + O(\sqrt{T} \log^{n+2} T).
\end{align*}

Recombining these, we have
\begin{align*}
\overline{I_L} & = (-1)^n \left\{ X \log^n X \left[ \frac{1}{2} \log X + \frac{\pi i}{4} \right] \sum_{m \leq \frac{T}{2 \pi X}} e^{-2 \pi i mX} + \frac{1}{2}X \log^n X \sum_{m \leq \frac{T}{2 \pi X}} e^{-2 \pi i mX} \log m \right. \\
&\qquad \qquad - X \sum_{mr \leq \frac{T}{2 \pi X}} e^{-2 \pi i mrX} \Lambda (r) \log^n (rX) \Bigg\} + O(\sqrt{T} \log^{n+2} T).
\end{align*}
Taking complex conjugates,
\begin{align}
I_L & = (-1)^n \left\{ X \log^n X \left[ \frac{1}{2} \log X - \frac{\pi i}{4} \right] \sum_{m \leq \frac{T}{2 \pi X}} e^{2 \pi i mX} + \frac{X}{2} \log^n X \sum_{m \leq \frac{T}{2 \pi X}} e^{2 \pi i mX} \log m \right. \nonumber \\
&\qquad \qquad - X \sum_{mr \leq \frac{T}{2 \pi X}} e^{2 \pi i mrX} \Lambda (r) \log^n (rX) \Bigg\} + O(\sqrt{T} \log^{n+2} T). \label{Left}
\end{align}

\subsection{Finalising the Proof of Theorem \ref{Explicit1}} \hfill\\
Combining (\ref{TopBot}), (\ref{Right}), and (\ref{Left}) gives $I$ in the second way that we are looking for. Specifically, we have
\begin{align*}
I =  (-1)^n &\left\{  \Delta(X) \frac{T}{2 \pi} \left( \log ^n X \left( \frac{1}{2} \log \left( \frac{T}{2 \pi} \right) - \frac{1}{2} + \frac{\pi i}{4} \right) - \sum_{mr=X} \Lambda (r) \log^n m  \right) \right.\\
&+ X \log^n X \left(\frac{1}{2} \log X - \frac{\pi i}{4} \right) \sum_{m \leq \frac{T}{2 \pi X}} e^{2 \pi i m X} \\
&+ \frac{X}{2} \log^n X \sum_{m \leq \frac{T}{2 \pi X}} e^{2 \pi i m X} \log m  \\
&- X \sum_{mr \leq \frac{T}{2 \pi X}}  e^{2 \pi i mr X} \Lambda (r) \log^n (rX) \Bigg\} \\
&+ O(\sqrt{T} \log^{n+2} T),
\end{align*}
Combining this asymptotic expansion with our observation that $I = \sum_{0 < \gamma \leq T}  \zeta^{(n)} (\rho) X^{\rho}$ in (\ref{Cauchy}) completes the proof of Theorem \ref{Explicit1}.

\section{Proof of Corollary \ref{IntegerCase}} \label{ProofCor}

We may rewrite Theorem \ref{Explicit1} in a slightly different way as follows.
\begin{corollary} \label{Explicit2}
For $X$ a fixed positive real number, we have
\begin{multline*}
\sum_{0< \gamma \leq T} \zeta^{(\lowercase{n})} (\rho)X^{\rho}  = \\(-1)^n \left\{  \Delta(X) \frac{T}{2 \pi} \left( \log ^n X \left( \frac{1}{2} \log \left( \frac{T}{2 \pi} \right) - \frac{1}{2} + \frac{\pi i}{4} \right) - \sum_{mr=X} \Lambda (r) \log^n m  \right) \right.\\
+ X \log^n X \left(\frac{1}{2} \log X - \frac{\pi i}{4} \right) \sum_{m \leq \frac{T}{2 \pi X}} e^{2 \pi i m X} + \frac{X}{2} \log^n X \sum_{m \leq \frac{T}{2 \pi X}} e^{2 \pi i m X} \log m  \\
- X \sum_{k=0}^{n} {n \choose k} \log^{n-k} X \left( \sum_{mr \leq \frac{T}{2 \pi X}} e^{2 \pi i mr X} \Lambda (r) \log^k r \right) \Bigg\} + O(\sqrt{T} \log^{n+2} T),
\end{multline*}
where
$\Delta (X)$ is given in Theorem \ref{Explicit1}.
\end{corollary}

\begin{proof}
This follows from Theorem \ref{Explicit1} by using the binomial expansion on $\log^{n} (rX)$ in the last summation in the braces in Theorem \ref{Explicit1}.
\end{proof}

\begin{remark}
The advantage to rewriting Theorem \ref{Explicit1} in the form of Corollary \ref{Explicit2} is that none of the summations involving exponentials have any reliance on powers of $\log X$.
\end{remark}

When $X \geq 1$ and $X \in \mathbb{Z}$, notice that
\[
\sum_{m \leq \frac{T}{2 \pi X}} e^{2 \pi i m X} = \frac{T}{2 \pi X} + O(1),
\]
\[
\sum_{m \leq \frac{T}{2 \pi X}} e^{2 \pi i m X} \log m = \frac{T}{2 \pi X} \log \left( \frac{T}{2 \pi X} \right) - \frac{T}{2 \pi X} + O(\log T)
\]
\[
\sum_{mr \leq \frac{T}{2 \pi X}} e^{2 \pi i mr X} \Lambda (r) \log^k r = \sum_{mr \leq \frac{T}{2 \pi X}} \Lambda (r) \log^k r.
\]
Define
\begin{equation}
S = (-1)^{k+1} \sum_{mr \leq \frac{T}{2 \pi X}} \Lambda (r) \log^k r \label{EquS}
\end{equation}
Then in an entirely analogous way to that done in Hughes and Pearce-Crump \cite{HPC}, we have
\begin{align*}
S & = (-1)^{k+1} \frac{1}{k+1} \frac{T}{2 \pi X} \log^{k+1} \left(\frac{T}{2 \pi X} \right) \\
& + (-1)^{k+1} \sum_{l=0}^{k} {k \choose l} (-1)^l l! \left( -1 + \sum_{j=0}^{l} (-1)^j C_j \right) \frac{T}{2 \pi X} \log^{k-l} \left( \frac{T}{2 \pi X} \right) + k! A_k \frac{T}{2 \pi X} + E_k(T)
\end{align*}
where
\[
E_k(T) =
\begin{cases}
O \left(T  e^{-C \sqrt{\log T}} \right) &\text{ unconditionally}\\
O \left(\sqrt{T} \log^{k+2} T \right) &\text{ under RH}\\
      \end{cases}
\]
and $C>0$ is a constant. Multiplying $S$ by $(-1)^{k+1}$ gives the summation that we were originally looking for.

Substituting these expressions into the asymptotic expansion from Corollary \ref{Explicit2} gives
\begin{align*}
\sum_{0< \gamma \leq T} \zeta^{(\lowercase{n})} (\rho)X^{\rho}  = (-1)^n &\left\{ \frac{T}{2 \pi} \left( \log ^n X \left( \frac{1}{2} \log \left( \frac{T}{2 \pi} \right) - \frac{1}{2} + \frac{\pi i}{4} \right) - \sum_{mr=X} \Lambda (r) \log^n m  \right) \right.\\
&+  \log^n X \left(\frac{1}{2} \log X - \frac{\pi i}{4} \right)  \frac{T}{2 \pi} \\
&+ \frac{1}{2} \log^n X \frac{T}{2 \pi} \log \left( \frac{T}{2 \pi X} \right) - \frac{T}{2 \pi} \\
&-  \sum_{k=0}^{n} {n \choose k} \log^{n-k} X \left( \frac{1}{k+1} \frac{T}{2 \pi} \log^{k+1} \left(\frac{T}{2 \pi X} \right) \right. \\
& + \sum_{l=0}^{k} {k \choose l} (-1)^l l! \left( -1 + \sum_{j=0}^{l} (-1)^j C_j \right) \frac{T}{2 \pi} \log^{k-l} \left( \frac{T}{2 \pi X} \right) \\
& + (-1)^{k+1} k! A_k \frac{T}{2 \pi} \Bigg) \Bigg\} + E_n(T),
\end{align*}
where
$E_n(T)$, $C_j$ and $A_k$ are defined in the statement of Corollary \ref{IntegerCase}.

We want to simplify this so that the $\log \left(\frac{T}{2 \pi X} \right)$ terms can be written separately in terms of $\log \left(\frac{T}{2 \pi} \right)$ and $\log X$. To do this, we begin by writing this previous expansion in a slightly different form and do some simplifying.
\begin{align*}
\sum_{0< \gamma \leq T} \zeta^{(\lowercase{n})} (\rho)X^{\rho} & = (-1)^n \frac{T}{2 \pi} \log ^n X \left( \log \left( \frac{T}{2 \pi} \right) -1 - \sum_{mr=X} \Lambda (r) \log^n m \right) \\
& + (-1)^{n+1} \sum_{k=0}^{n} {n \choose k} \log ^{n-k} X \Bigg\{ \frac{1}{k+1} \frac{T}{2 \pi X} \left( \log \left( \frac{T}{2 \pi} \right) - \log X \right)^{k+1} \\
&\qquad + \sum_{l=0}^{k} {k \choose l} (-1)^l l! \left( -1 + \sum_{j=0}^{l} (-1)^j C_j \right) \frac{T}{2 \pi} \left( \log \left( \frac{T}{2 \pi} \right) - \log X \right)^{k-l} \\
&\qquad + (-1)^{k+1} k! A_k \frac{T}{2 \pi}  \Bigg\} + E_n(T).
\end{align*}

Applying the binomial theorem, we have
\begin{align*}
\sum_{0< \gamma \leq T}& \zeta^{(\lowercase{n})} (\rho) X^{\rho}  = (-1)^n \frac{T}{2 \pi} \log ^n X \left( \log \left( \frac{T}{2 \pi} \right) -1 - \sum_{mr=X} \Lambda (r) \log^n m \right) \\
& + (-1)^{n+1} \frac{T}{2 \pi} \Bigg\{ \sum_{k=0}^{n} \sum_{u=0}^{k+1} {n \choose k} {{k+1} \choose u} (-1)^u \frac{1}{k+1}  \log^{k+1-u} \left( \frac{T}{2 \pi} \right)\log^{n-k+u} X \\
& + \sum_{k=0}^{n}  \sum_{l=0}^{k} \sum_{u=0}^{k-l}  {n \choose k} {k \choose l} {{k-l} \choose u} (-1)^l l! \left( -1 + \sum_{j=0}^{l} (-1)^j C_j \right) \log^{k-l-u} \left( \frac{T}{2 \pi} \right) \log^{n-k+u} X \\
& + \sum_{k=0}^{n} {n \choose k} (-1)^{k+1} k! A_k \log ^{n-k} X \Bigg\} + E_n(T).
\end{align*}

Finally, we may rearrange slightly to complete the proof of Corollary \ref{IntegerCase} to give
\begin{align*}
\sum_{0< \gamma \leq T}& \zeta^{(\lowercase{n})} (\rho) X^{\rho}  = (-1)^{n+1} \frac{T}{2 \pi} \log ^n X \Bigg\{  \sum_{k=0}^{n} \sum_{u=0}^{k+1} {n \choose k} {{k+1} \choose u} (-1)^u \frac{1}{k+1}  \log^{k+1-u} \left( \frac{T}{2 \pi} \right)\log^{u-k} X \\
& + \sum_{k=0}^{n}  \sum_{l=0}^{k} \sum_{u=0}^{k-l}  {n \choose k} {k \choose l} {{k-l} \choose u} (-1)^l l! \left( -1 + \sum_{j=0}^{l} (-1)^j C_j \right) \log^{k-l-u} \left( \frac{T}{2 \pi} \right) \log^{u-k} X \\
& + \sum_{k=0}^{n} {n \choose k} (-1)^{k+1} k! A_k \log ^{-k} X - \left( \log \left( \frac{T}{2 \pi} \right) -1 - \sum_{mr=X} \Lambda (r) \log^n m \right) \Bigg\} + E_n(T).
\end{align*}

\section*{\textbf{Acknowledgements}}
The author is grateful to his supervisor, Dr. Christopher Hughes, for all of his help and guidance when writing this paper, as well as introducing him to the original paper that inspired this generalisation.

\begin{remark}
This paper will form part of the author's PhD thesis at the University of York.
\end{remark}

\end{document}